\documentclass{amsart}
\usepackage{microtype,amssymb,algorithm,algpseudocode,graphicx,color,cite,amsthm,enumitem}
\usepackage{hyperref}
\usepackage[dvipsnames]{xcolor}

\def\C{{\mathbb C}}
\def\RR{{\mathbb R}}

\def\?{{\bf ??}}

\def\dim{{\rm dim}}

 \newtheorem{theorem}{Theorem}[section]
\newtheorem{lemma}[theorem]{Lemma} 
 
\newtheorem{definition}[theorem]{Definition} 
\newtheorem{corollary}[theorem]{Corollary}
\newtheorem*{claim*}{Claim}
 
\newtheorem{remark}[theorem]{Remark}

\newtheorem{example}[theorem]{Example}  

\numberwithin{equation}{section}

\def\CC{\mathbb C}
\def\rank{\mathrm{rank}}
\def\f{\frac}

\newcommand{\Diag}{\mbox{\rm Diag}\,}
\newcommand{\FF}{{\mathbb{F}}}
\newcommand{\VV}{{\mathcal{V}}}
\newcommand{\EE}{{\mathcal{E}}}
\newcommand{\TT}{{\mathcal{T}}}
\newcommand{\MM}{{\mathcal{M}}}

\newcommand{\wt}{\widetilde}

\newcommand{\ra}{\rightarrow}

\newcommand{\lt}{\left}
\newcommand{\rt}{\right}
\newcommand{\ol}{\overline}

\newcommand{\parallelslant}{\mathbin{\!/\mkern-5mu/\!}}
  
\date{\today}

\title[ED degree of orthogonally invariant varieties]{The Euclidean Distance Degree of Orthogonally Invariant Matrix Varieties}
\author[D. Drusvyatskiy \and H.L. Lee \and G. Ottaviani \and R.R. Thomas]
{Dmitriy Drusvyatskiy \and Hon-Leung Lee \and Giorgio Ottaviani \and Rekha R. Thomas}
\address{Department of Mathematics, University of Washington, Box 354350, Seattle, WA 98195-4350}
\email{[ddrusv, hllee, rrthomas]@uw.edu}
\address{*Universit\`a di Firenze, viale Morgagni 67A, 50134 Firenze,
Italy}
\email{ottavian@math.unifi.it}
\thanks{Drusvyatskiy was partially supported by the AFOSR YIP award FA9550-15-1-0237. Lee and Thomas were partially supported by the NSF grant DMS-1418728.}

\begin{document}

\begin{abstract} 
	We show that the Euclidean distance degree of a real orthogonally invariant matrix variety equals the 
	Euclidean distance degree of its restriction to diagonal matrices. 
We illustrate how this result can greatly simplify calculations in concrete circumstances.
\end{abstract}

\maketitle
\section{Introduction}
The problem of minimizing the Euclidean distance (ED) of an observed data point $y \in \RR^n$ to a real algebraic variety $\mathcal{V} \subseteq \RR^n$ arises frequently in applications, and amounts to solving the polynomial optimization problem
$$\textrm{minimize} \quad \sum^n_{i=1} (y_i-x_i)^2 \quad \textrm{ subject to }\quad x\in \mathcal{V}.$$
The algebraic complexity of this problem is closely related to the number of complex regular critical points of $y$ on the Zariski closure $\mathcal{V}_\CC$ of $\mathcal{V}$. We will call such points the {\em ED critical points} of $y$ with respect to $\mathcal{V}$; see Definition~\ref{defn:edcrit}. The authors of \cite{DHOST} showed that the number of 
ED critical points of a general data point $y\in \CC^n$ is a constant, and  hence is an invariant of $\mathcal{V}$. This number is called the {\em Euclidean distance degree} (ED degree) of $\mathcal{V}$. As noted in \cite{DHOST}, the computation of $\mathrm{EDdegree}(\mathcal{V})$ can be subtle, since it may change considerably under a linear transformation of $\mathcal{V}$.

In this work, we explore the ED degree of {\em orthogonally invariant} matrix varieties $\mathcal{M} \subseteq \RR^{n \times t}$, meaning those varieties $\mathcal{M}$ satisfying
$$U\mathcal{M}V^\top = \mathcal{M}\quad \textrm{ for all real orthogonal matrices}\quad U \in O(n), V \in O(t).$$
Without loss of generality, suppose $n\leq t$.
Clearly, membership of a matrix $M$ in such a variety $\mathcal{M}$ is fully determined by its vector of singular values $\sigma(M)=(\sigma_1(M), \ldots, \sigma_n(M))$, where we use the convention
$\sigma_{i-1}(M) \geq \sigma_i(M)$ for each $i$. 
Indeed, we may associate with any orthogonally invariant matrix variety $\mathcal{M}$ its {\em diagonal restriction} $S=\{x:\Diag(x)\in \mathcal{M}\}$. The variety $S$ thus defined is {\em absolutely symmetric} (invariant under signed permutations) and satisfies the key relation $\mathcal{M}=\sigma^{-1}(S)$. Conversely, any absolutely symmetric set $S\subseteq \RR^n$ yields the orthogonally invariant matrix variety $\sigma^{-1}(S)$; see e.g. \cite[Theorem~3.4]{DLT} and \cite[Proposition 1.1]{man}. 

In this paper, we prove the elegant formula
\begin{equation}\label{ED_res}
\boxed{\mathrm{EDdegree}(\MM) = \mathrm{EDdegree}(S) \tag{$\star$}.}
\end{equation}
In most interesting situations, the diagonal restriction $S\subseteq\RR^n$ has simple geometry, as opposed to the matrix variety $\mathcal{M}$, and hence our main result \eqref{ED_res} provides elementary and transparent means to compute the ED degree of $\MM$ by working with the simpler object $S$. 
Interesting consequences flow from there. For example, consider 
the {\em $r$-th rank variety} 
$$\RR^{n\times t}_r:=\{X\in\RR^{n\times t}: \rank\, X\leq r\},$$
and the {\em essential variety}
$$\mathcal{E}:=\{X\in \RR^{3\times 3}: \sigma_1(X)=\sigma_2(X),\, \sigma_3(X)=0\}$$
from computer vision \cite{epipolarPaper_v2,MaybankBook, hartley-zisserman-2003}; both $\RR^{n\times t}_r$ and  $\mathcal{E}$ are orthogonally invariant. The diagonal restrictions of $\RR^{n\times t}$ and $\mathcal{E}$ are finite unions of linear subspaces and their ED degrees are immediate to compute. Moreover, our results readily imply that  all ED critical points of a general real data matrix $Y$ on $\RR^{n\times t}_r$ and on $\mathcal{E}$ are real. This result has been previously shown for $\RR^{n\times t}_r$ in \cite{DHOST} -- a generalization of the Eckart-Young theorem -- and is entirely new for the essential variety $\mathcal{E}$. A related further investigation of the essential variety appears in \cite{FKO}.

Our investigation of orthogonally invariant matrix varieties fits in a broader scope.
The idea of studying orthogonally invariant matrix sets $ \MM$ via their diagonal restrictions $S$ -- the theme of our paper -- is not new, and goes back at least to von Neumann's theorem on {\em unitarily invariant matrix norms} \cite{von_Neumann}. In recent years, it has become clear that various analytic properties of $\MM$ and $S$ are in one-to-one correspondence, and this philosophy is sometimes called the ``transfer principle''; see for instance, \cite{man}. For example, $\MM$ is $C^p$-smooth around a matrix $X$ if and only if $S$ is $C^p$-smooth around $\sigma(X)$ \cite{spec_id,der,diff_2,diff_1,dc_new}. Other properties, such as {\em convexity} \cite{cov_orig}, {\em positive reach} \cite{spec_prox}, {\em partial smoothness} \cite{spec_id}, and {\em Whitney conditions} \cite{mather} follow the same paradigm. In this sense, our paper explores the transfer principle for the ED degree of algebraic varieties. To the best of our knowledge, this is the first result in this body of work that is rooted in algebraic geometry. 


Though our main result \eqref{ED_res} is easy to state, the proof is subtle;  moreover, the result itself is surprising in light of the discussion in \cite[Section 5]{DLT}. The outline of the paper is as follows.
In Section~\ref{sec:complexification} we investigate invariance properties of the Zariski closure $\MM_\CC \subseteq \CC^{n \times t}$ of an orthogonally invariant matrix variety $\MM\subseteq \RR^{n \times t}$, as well as the correspondence between irreducible components of $S$ and those of $\MM$.
In Section~\ref{sec:algebraicSVD}, we discuss ``algebraic singular value decompositions'' for general matrices $Y \in \CC^{n \times t}$, leading to
Section~\ref{sec:complexED} containing our main results. When $S$ is a subspace arrangement, our results yield particularly nice consequences generalizing several classical facts in matrix theory -- the content of Section~\ref{subsec:subspaceArrangements}.

\section{Zariski closure, irreducibility, and dimension of matrix varieties }\label{sec:complexification}
Setting the stage, we begin with some standard notation. For the fields $\FF = \RR$ or $\FF =\CC$, the symbol $\FF[x]= \FF[x_1,\ldots,x_n]$ will denote the ring of polynomials in $x_1,\ldots,x_n$ with coefficients in $\FF$. 
Given polynomials $f_1,\ldots,f_s\in \FF[x]$ the set $\VV:=\{x\in \FF^n \ : \  f_1(x) = \cdots = f_s(x)=0\} $ is called an (algebraic) {\em variety} over $\FF$.
The {\em Zariski closure} of an arbitrary set $T$ in $\CC^n$, denoted $\ol{T}$,  is the smallest variety over $\CC$ containing $T$.  
Unless otherwise specified, the topology on $\CC^n$ is fixed to be the {\em Zariski topology}, obtained by defining the closed sets to be  the varieties over $\CC$.
The topology on any subset of $\CC^n$ will then always be the one induced by the Zariski topology.

Consider a real algebraic variety $\mathcal{V} \subseteq \RR^n$.
 The {\em vanishing ideal} of $\mathcal{V}$ is defined to be $I(\mathcal{V}) := \{ f \in \RR[x] \,:\, f(x) = 0 \,\, \text{ for all }\, x \in \mathcal{V} \}$.
Viewing $\mathcal{V}$ as a subset of $\CC^n$, the Zariski closure of $\VV$, denoted $\VV_\CC$, 
can be written as $\{x\in \CC^n \,:\, f(x)=0\text{ for all } f\in I(\VV)\}$; see e.g. \cite{whitney}. Note this notation is slightly redundant since by definition we have $\overline{\mathcal{V}}=\VV_\CC$. Nonetheless, we prefer to keep both symbols to ease notation when appropriate. 



\subsection{Invariance under closure}\label{subsec:inv}
Consider a group $G$ acting linearly on $\CC^n$. Then  $G$ also acts on $\CC[x]$ via 
$$
g \cdot f(x) = (x \mapsto f(g\cdot x))\qquad \text{ for any } g\in G, \ f\in \CC[x].
$$
A subset $T\subseteq \CC^n$ is {\em $G$-invariant} if $g\cdot x$ lies in $T$ for any $g\in G$ and $x\in T$. 
A polynomial $f\in \CC[x]$ is {\em $G$-invariant} provided $g\cdot f = f$ for all $g\in G$. 
We begin with the following elementary result.
\begin{lemma}\label{lem:closure is Ginvariant}
 If a set $T\subseteq \CC^n$ is $G$-invariant, then its closure
$\ol{T}$ is also $G$-invariant.
\end{lemma}

\begin{proof}
Fixing $g\in G$, the map $\mu_g: \CC^n \ra \CC^n$ given by $\mu_g(x)=g\cdot x$ is a linear isomorphism. Hence, assuming $T$ is $G$-invariant, we deduce $\mu_g(\ol{T}) = \ol{\mu_g(T)} = \ol{T}$, as claimed.
\end{proof}

We now specialize the discussion to the main setting of the paper. 
For a positive integer $s$, the symbol $O(s)$ will denote the set of all $s \times s$ real orthogonal matrices. This is both a group and a real variety and its Zariski closure $O_{\CC}(s)$ is the set of all $s \times s$ complex orthogonal matrices --- those satisfying $Q^\top  Q=QQ^\top =I$. 
Henceforth, we fix two positive integers $n$ and $t$ with $n \leq t$, and consider the groups $O(n)\times O(t)$ and 
$O_{\CC}(n)\times O_{\CC}(t)$, along with the group $\Pi_n^\pm$ of all signed permutations of $\{1,\ldots,n\}$. Recall that we always consider the action  of $O(n)\times O(t)$ on $\RR^{n\times t}$ and the action of $O_{\CC}(n)\times O_{\CC}(t)$ on $\CC^{n\times t}$ by conjugation $(U,V)\cdot X=UXV^\top $.

Now suppose $\MM \subseteq \RR^{n \times t}$ is a $O(n) \times O(t)$-invariant (orthogonally invariant) matrix variety. Then Lemma \ref{lem:closure is Ginvariant} shows that 
$\MM_\CC$ is $O(n) \times O(t)$-invariant. We now prove the stronger statement: 
$\MM_\CC$ is invariant under the larger group  $O_{\CC}(n) \times O_{\CC}(t)$. 

\begin{theorem}[Closure invariance] \label{thm:invariant varieties}
A matrix variety $\MM \subseteq \RR^{n \times t}$ is  $O(n) \times O(t)$-invariant if and only if $\MM_\CC$ is $O_{\CC}(n) \times O_{\CC}(t)$-invariant. 
Similarly, a variety $S \subseteq \RR^n$ is  $\Pi_n^\pm$-invariant if and only if $S_\CC$ is $\Pi_n^\pm$-invariant.
\end{theorem}

\begin{proof}
Since the proofs are similar, we only prove the first claim. The backward implication is trivial. Suppose $\MM $ is  $O(n) \times O(t)$-invariant. 
Let $X\in \MM_\CC$ be fixed. Then the map $\gamma_X: O_\CC(n) \times O_\CC(t)  \ra \CC^{n\times t}$ defined by $\gamma(g) = g\cdot X$ is continuous. 
Lemma \ref{lem:closure is Ginvariant} yields the inclusion $O(n) \times O(t)\subseteq \gamma_X^{-1}(\MM_\CC)$. Since $\gamma_X^{-1}(\MM_\CC)$ is closed by continuity, we conclude $O_\CC(n)\times O_\CC(t)  
\subseteq
 \gamma_X^{-1}(\MM_\CC)$. This completes the proof.
\end{proof}

\subsection{Irreducible components of orthogonally invariant varieties}\label{subsec:components}
For the rest of the section, fix a $\Pi^{\pm}_n$-invariant (absolutely symmetric) variety $S$ in $\RR^n$. Then the 
$O(n) \times O(t)$-invariant matrix set $\mathcal{M}:=\sigma^{-1}(S)$ is a real variety in $\RR^{n\times t}$; see \cite[Theorem 3.4]{DLT} or  
\cite[Proposition 1.1]{man}. Moreover, the diagonal restriction  $\{ x \in \RR^n \,:\, \Diag(x) \in \MM \}$
coincides with $S$.
Here, we call a $n \times t$ matrix $D$ diagonal if and only if $D_{ij}=0$ whenever $i \neq j$, and
for any vector $x\in \RR^n$ the symbol $\Diag(x)$ denotes the diagonal matrix with $D_{ii}=x_i$ for each $i=1,\ldots,n$.

In this section, we highlight the correspondence between the irreducible components of $S$ and those of $\MM$. Recall that 
a real or complex variety $\mathcal{V}$ is {\em irreducible} 
if it cannot be written as a union $\mathcal{V}=\mathcal{V}_1\cup\mathcal{V}_2$ of two proper subvarieties $\mathcal{V}_1$ and $\mathcal{V}_2$.
Any variety $\mathcal{V}$ can be written as a union of finitely many irreducible subvarieties $\mathcal{V}_i$ satisfying $\mathcal{V}_i\nsubseteq \mathcal{V}_j$ for distinct indices $i$ and $j$. The varieties $\mathcal{V}_i$ are called the {\em irreducible components} of $\mathcal{V}$, and are uniquely defined up to indexing.

Coming back to the aim of this section, let $\{S_i\}^k_{i=1}$ be the irreducible components of $S$. The varieties $S_i$ are typically not absolutely symmetric. 
Hence we define their symmetrizations $S^{\pi}_i:=\bigcup_{\pi \in \Pi^{\pm}_n} \pi S_i$ and the real varieties $\mathcal{M}_i:=\sigma^{-1}(S^{\pi}_i)$.
It is standard that a signed permutation maps an irreducible component of $S$ to another irreducible component of $S$.

We record the following elementary observation for ease of reference.

\begin{lemma}\label{lem:incl}
For any pair of indices $i,j$, the following implications hold:
\begin{align*}
S_{i}^{\pi}\subseteq S_{j}^{\pi} \quad &\Longrightarrow \quad S_{i}^{\pi}= S_{j}^{\pi}, \textup{ and } \\
\MM_i\subseteq \MM_j \quad &\Longrightarrow \quad \MM_i= \MM_j
\end{align*}
\end{lemma}
\begin{proof}
If $S_{i}^{\pi}\subseteq S_{j}^{\pi}$, then we deduce that $S_i= \bigcup_{\pi\in \Pi^{\pm}_n} (S_i\cap \pi S_j)$. Hence for some $\pi\in \Pi^{\pm}_n$, the inclusion $S_i\subseteq \pi S_j$ holds. Since both $S_i$ and $\pi S_j$ are irreducible components of $S$, it must be that $S_i= \pi S_j$ and hence, $S_{i}^{\pi}= S_{j}^{\pi}$, as claimed. The second implication follows immediately.
\end{proof}

For any $U\in O(n)$ and $V\in O(t)$, the map 
$X\mapsto UXV^\top $  is an automorphism of $\MM_i$, and therefore maps an irreducible component of $\mathcal{M}_{i}$ to another irreducible component of $\mathcal{M}_{i}$. We now show that this action is transitive, just as the action of $\Pi^{\pm}_n$ on the components of $S_i^{\pi}$. 



\begin{lemma}\label{lem:tran}
For any index $i$, the group $O(n)\times O(t)$ acts
transitively on the irreducible components of $\MM_i$. Consequently, the real variety $\MM_i$ is equidimensional.
\end{lemma}
\begin{proof}
Let  $\mathcal{H}$ be an irreducible component of $\mathcal{M}_{i}$. 
	Note that the set
	 $\Gamma:=\bigcup\{U\mathcal{H}V^\top : U\in O(n), V\in O(t)\}$ is a union of some irreducible components of $\mathcal{M}_{i}$. Let $Z$ be the union of the irreducible components of $\mathcal{M}_{i}$ not contained in $\Gamma$ (if any). Observe that $Z$ is an orthogonally invariant variety. Hence the two absolutely symmetric varieties $\{x: \Diag(x)\in \Gamma\}$ and $\{x: \Diag(x)\in Z\}$ cover $S^{\pi}_i$.
	Since $S_i$ is irreducible, either $\Gamma$ or $Z$ coincides with all of $\MM_i$. Since the latter is impossible by construction, we conclude that $\Gamma=\MM_i$, as claimed.
\end{proof}

We end with the following theorem, which will play a key role in the proof of Lemma~\ref{transfer of tangent spaces}, leading to the main result of the paper.

\begin{theorem}\label{thm:union_comp}
Each variety $\mathcal{M}_{i}$ is a union of some irreducible components of $\mathcal{M}$.
\end{theorem}
\begin{proof}
Let $\{C_l\}$ be the set of  irreducible components of  $\MM$. Then for any index $l$,  $C_l=\bigcup_j (C_l\cap\mathcal{M}_{j})$ which implies that $C_l$ is contained in $\mathcal{M}_{j}$ for some index $j(l)$. 


Fix an $\MM_i$
and let $\mathcal{H}$ be an irreducible component of $\MM_i$. From the equality $\mathcal{H}=\bigcup_l  (C_l\cap \mathcal{H})$ we conclude that the inclusion $\mathcal{H}\subseteq C_l$ holds for some index $l$. This implies that $\mathcal{H}\subseteq C_l\subseteq \MM_{j(l)}$, and hence by Lemma~\ref{lem:tran}, $\MM_i\subseteq \MM_{j(l)}$. Lemma~\ref{lem:incl} then implies the equality $\MM_i=\MM_{j(l)}$, yielding $\mathcal{H}\subseteq C_{l}\subseteq \MM_{i}$. Taking the union of this inclusion over all the irreducible components $\mathcal{H}$ of $\MM_{i}$ and the corresponding $C_l$, we deduce that $\MM_i$ is a union of some irreducible components of $\MM$. 
\end{proof}

Since closures of irreducible components of a real variety $\mathcal{V}$ are the irreducible components of $\mathcal{V}_{\CC}$, Theorem~\ref{thm:union_comp} immediately implies that $\overline{\MM_i}$ is a union of some irreducible components of $\MM_{\CC}$.


\subsection{Dimension of orthogonally invariant varieties}\label{subsec:dim}
We next show how to read off the dimension of $\MM_\CC$ from the absolutely symmetric variety $S\subseteq \RR^n$. 
To this end, note first that since the equality 
$\textup{dim}(\MM_\CC) = \textup{dim}(\MM)$ holds (see \cite[Lemma 8]{whitney}), it suffices to compute the dimension of the real variety $\MM$ from $S$. We will  assume that $\Pi^{\pm}_n$ acts transitively on the irreducible components of $S$, that is in the notation of Section \ref{subsec:components} we have $S_i^{\pi}=S$ for all indices $i$. If this is not the case, we can treat each set $S_i^{\pi}$ separately. With this simplification, both varieties $S$ and $\MM$ are equidimensional (Lemma~\ref{lem:tran}). 


The following recipe follows that in \cite[Section 2.3]{spec_id} and \cite{man} and hence we skip some of the explanations. 
The basic idea is to understand the dimension of the fiber $\sigma^{-1}(x^*)$ where $x^*\in S$ is (carefully) chosen so that the sum
of the dimension of the fiber and the dimension of $S$  equals $\dim(\MM)$.

Fixing notation, consider  the convex cone
$$\RR^n_{+,\geq}:=\{x\in\RR^n: x_1\geq x_2\geq\ldots\geq x_n\geq 0\}.$$
Observe that $\RR^n_{+,\geq}$ is exactly the range of $\sigma$ on $\RR^{n \times t}$.
Along with a point $x \in \RR^n_{+,\geq}$, we associate the partition $\mathcal{P}_x = \{P_1, \ldots, P_{\rho_x}, P_0\}$ of the index set $\{1,\ldots,n\}$ so that 
$x_i = x_j$ if and only if $i,j \in P_l$, and 
$x_i >x_j$ for any $i\in P_q$ and $j\in P_r$ with $q >r$. We assume that $P_0$ contains the indices of the zero coordinates in $x$, and we define 
$p_l := | P_l |$. It could be that $p_0=0$ for a given $x$. On the other hand, we have $p_l > 0$ for all $l=1,\ldots,\rho_x$.
Recall the equality 
\begin{align*} 
\sigma^{-1}(x) = \{ U \, \Diag(x) \, V^\top  \,:\, U \in O(n), V \in O(t) \}.
\end{align*}
Let $$(O(n) \times O(t))_x := \{ (U,V) \in O(n) \times O(t) \,:\, \Diag(x) = U \, \Diag(x) \, V^\top  \}$$ denote the stabilizer of $\Diag(x)$, under the action of $O(n) \times O(t)$. Then one can check that $(U,V)$ lies in the stabilizer $(O(n) \times O(t))_x$ if and only if 
$U$ is block diagonal with blocks $U_i \in O(p_i)$ for $i=0,\ldots,\rho_x$ and $V$ is block diagonal with blocks $V_i \in O(p_i)$ for $i=1,\ldots,\rho_x$, and a block $V_0 \in O(p_0 + (t-n))$. Further, $U_i V_i^\top  = I$ for all $i=1,\ldots, \rho_x$ which means that the $U_i$'s determine the corresponding $V_i$'s for all $i$ except $i=0$. This implies that the dimension of $(O(n) \times O(t))_x$ is 
\begin{align*}
\textup{dim}((O(n) \times O(t))_x ) = \sum_{l=0}^{\rho_x} \frac{p_l (p_l-1)}{2}  + \frac{(p_0+t-n)(p_0+t-n-1)}{2}
\end{align*}
yielding 
\begin{align}\label{eq:dim orbit}
\textup{dim}(\sigma^{-1}(x)) & =  \textup{dim}(O(n) \times O(t)) - \textup{dim}((O(n) \times O(t))_x )\nonumber\\
&= \frac{n(n-1)+t(t-1)}{2}-\sum_{l=0}^{\rho_x} \frac{p_l (p_l-1)}{2}  - \frac{(p_0+t-n)(p_0+t-n-1)}{2}\nonumber\\
& =  \sum_{0 \leq i < j \leq \rho_x}  p_i p_j  + \frac{t(t-1)}{2}  - \frac{(p_0+t-n)(p_0+t-n-1)}{2}.
\end{align}
Here we used the observation $$ \frac{n(n-1)}{2}-\sum_{l=0}^{\rho_x} \frac{p_l (p_l-1)}{2}={\sum_{l=0}^{\rho_x} p_l \choose 2}-\sum_{l=0}^{\rho_x}{p_l\choose 2}=\sum_{0 \leq i < j \leq \rho_x}  p_i p_j.$$

 For a partition $\mathcal{P}$ of $[n]$, define the set $\Delta_{\mathcal{P}} := \{ x \in \RR^n_{+,\geq} \,:\, \mathcal{P}_x = \mathcal{P} \}$. The set of all such $\Delta$'s defines an affine stratification of $\RR^n_{+,\geq}$. Let $\mathcal{P}_\ast$ correspond to a stratum $\Delta_\ast$ in this stratification satisfying $S \cap \Delta_\ast \neq \emptyset$ and having maximal dimension among all strata that have a nonempty intersection with $S$. Then for any point $x^\ast \in S \cap \Delta_\ast$, we can choose a sufficiently small $\delta>0$ satisfying
 $S\cap B_{\delta}(x^\ast)\subseteq \Delta_\ast$. Hence the fibers $\sigma^{-1}(x)$ have the same dimension for all $x\in S\cap B_{\delta}(x^\ast)$ and the preimage  $\sigma^{-1}(S\cap B_{\delta}(x^\ast))$ is an open (in the Euclidean topology) subset of $\MM$. Taking into account that both $S$ and $\MM$ are equidimensional, we deduce 
 $$\dim(\sigma^{-1}(S))=\dim( S)+\dim (\sigma^{-1}(x^\ast)).$$
 Appealing to \eqref{eq:dim orbit}, we arrive at the formula 
\begin{align} \label{eq:dim formula}
\dim(\MM)=\textup{dim}(S) + \left( \sum_{0 \leq i < j \leq \rho^\ast}  p_i^\ast p_j^\ast  \right) + \frac{t(t-1)}{2}  - \frac{(p_0^\ast+t-n)(p_0^\ast+t-n-1)}{2}.
\end{align}

\begin{example} [Rank variety]
	{\rm
Recall the rank variety $\RR^{n \times t}_r$ of matrices of rank at most $r$. 
	In this case, $S$ is the union of all coordinate planes in $\RR^n$ of dimension $r$ and $S_\CC$ is the set of all $r$-dimensional coordinate planes in $\CC^n$. Also, $\MM_\CC = \CC^{n \times t}_r$, the set of all matrices in $\CC^{n \times t}$ of rank at most $r$.

Note that $S$ is equidimensional. Then along with a point $x^\ast$ we have $p_0^\ast = n-r$ and $p_i^\ast = 1$ for all $i=1, \ldots,r$. Applying \eqref{eq:dim formula} we get that the dimension of $\CC^{n \times t}_r$ is 
$$ r + \left({r \choose 2} + r(n-r) \right)+ \frac{t(t-1)}{2} - \frac{(t-r)(t-r-1)}{2} = r(t+n-r).$$}
\end{example}

\begin{example}[Essential variety] \label{ex:essential variety}
	{\rm
The essential variety is $\EE = \{ E \in \RR^{3 \times 3} \,:\, \sigma_1(E) = \sigma_2(E), \,\,\sigma_3(E) = 0 \}$. Its Zariski closure $\EE_\CC \subseteq \CC^{3 \times 3}$ is known to be irreducible and of dimension six \cite{demazure}. In this case, $S \subseteq \RR^3$ consists of the six lines defined by 
$x_1=\pm x_2$, $x_1=\pm x_3$ and $x_2 = \pm x_3$ with the remaining coordinate set to zero in each case.

We can verify $\dim(\EE_\CC) =6$ using \eqref{eq:dim formula}. Indeed, picking a generic point $x^\ast$ on the line $x_1=x_2$ in $\RR^3_+$, we see that $\mathcal{P}_{x^\ast}$ has $p_0^\ast = 1$ and $p_1^\ast = 2$. Now applying the formula \eqref{eq:dim formula} we get  
$\textup{dim}(\mathcal{E}_\CC) = 1 + 1 \cdot 2 + 3 - 0 = 6$.}
\end{example}


\section{Algebraic Singular Value Decompositions and GIT quotients} \label{sec:algebraicSVD}
In this section we fix a $\Pi^{\pm}_n$-invariant variety $S\subseteq\RR^n$ and the induced $O(n) \times O(t)$-invariant matrix variety $\mathcal{M}:=\sigma^{-1}(S)$. The description of $\mathcal{M}$ as the preimage $\sigma^{-1}(S)$ is not convenient when seeking to understand the algebraic geometric  correspondences between $\mathcal{M}$ and $S$, since $\sigma$ is not a polynomial map. Instead, we may equivalently write 
\begin{equation}\label{eqn:real_desc}
\MM = \{U \, \Diag(x) \, V^\top  \ : \   U\in O(n), \  V\in O(t), \  x\in S\}.
\end{equation}
In this notation, it is clear that $\MM$ is obtained from $S$ by an algebraic group action -- a description that is more amenable to analysis. Naturally then to understand geometric correspondences between the closures $\MM_{\CC}$ and $S_{\CC}$, we search for a description analogous to \eqref{eqn:real_desc}, with $\MM$, $S$, $O(n)$, and $O(t)$ replaced by their Zariski closures $\MM_{\CC}$, $S_{\CC}$, $O_{\CC}(n)$, and $O_{\CC}(t)$. The difficulty is that an exact equality analogous to 
\eqref{eqn:real_desc}
usually fails to hold; instead, equality holds only in a certain generic sense that is sufficient for our purposes. We now make this 
precise.

\subsection{Algebraic SVD}
Our strategy revolves around an ``algebraic singular value decomposition'', a notion to be made precise shortly. Note that the common extension of a singular value decomposition (SVD) from real to complex matrices 
using unitary matrices, their conjugates, and the Hermitian metric does not fit well in the algebraic setting because unitary matrices form a real (but not a complex) variety and conjugation is not an algebraic operation. In particular, it is not suitable for studying the $\mathrm{EDdegree}$ of a matrix variety. 
Hence we will need an algebraic analog of SVD that uses complex orthogonal matrices. For a recent geometric treatment of SVD rooted in algebraic geometry see the survey \cite{SVD_surv}.

\begin{definition}[Algebraic SVD]
{\rm
We say that a matrix $A \in \CC^{n \times t}$ admits an {\em algebraic SVD} if it can be factored as $A = U D V^\top $ for some orthogonal matrices $U\in O_\CC(n)$ and $V\in 
O_\CC(t)$, and a complex diagonal matrix $D\in\CC^{n\times t}$. }
\end{definition}

Not all matrices admit an algebraic SVD; indeed, this is the main obstruction to an equality analogous to  \eqref{eqn:real_desc} in which the varieties $\MM$, $S$, $O(n)$, and $O(t)$ are replaced by their closures.
A simple example is the matrix $A={\left( {\begin{smallmatrix}1&i\\0&0\end{smallmatrix}}\right)}$, with  $i=\sqrt{-1}$. Indeed, 
in light of the equality $AA^\top =0$, if it were possible to write $A=UDV^\top $ for some $U,V \in O_\CC(2)$ and a diagonal matrix $D$, then we would deduce that $UDD^\top  U^\top =0$ which implies that $A=0$, a contradiction. 
Fortunately, the existence question has been completely answered by Choudury and Horn \cite[Theorem~2 \& Corollary~3]{choudhury1987analog}. 

\begin{theorem}[Existence of an algebraic SVD]\label{thm:algebraic SVD}
A matrix $A\in \CC^{n\times t}$ admits an algebraic SVD, if and only if, 
$AA^\top $ is diagonalizable and $\rank(A) = \rank(AA^\top )$. 
\end{theorem}

Suppose $A$ admits an algebraic SVD  $A = U \,\Diag(d)\, V^\top $ for some orthogonal matrices $U\in O_\CC(n)$ and $V\in 
O_\CC(t)$, and a vector $d\in\CC^n$. Then the numbers $d^2_i$ are eigenvalues of $A^\top  A$ and $AA^\top $, and 
the columns of $U$ are eigenvectors of $AA^\top $ and the columns of $V$ are eigenvectors of $A^\top  A$, arranged in the same order as $d_i.$ We call the complex numbers $d_i$ the {\it algebraic singular values} of $A$. They are determined up to sign.

We record the following immediate consequence of Theorem~\ref{thm:algebraic SVD} for ease of reference.
\begin{corollary} \label{cor:algebraicSVD as we had}
A matrix $A \in \CC^{n \times t}$ has an algebraic SVD provided 
the eigenvalues of $AA^\top $ are nonzero and distinct.
\end{corollary}


Suppose $\VV$ is a variety over $\RR$ or $\CC$. We say that a property holds for a {\em generic point $x\in \VV$} if the set of points $x\in \VV$ for which the property holds contains an open dense subset of $\VV$ (in Zariski topology). In this terminology, Theorem~\ref{thm:algebraic SVD} implies that generic complex matrices $A \in \CC^{n \times t}$ do admit an algebraic SVD. 

We can now prove the main result of this section (cf. equation~\eqref{eqn:real_desc}).

\begin{theorem}[Generic description] \label{thm:generic points in M_C}
Suppose that a set $Q\subseteq S_\CC$ contains an open dense subset of $S_\CC$. Consider the set 
$$
\mathcal{N}_Q:=\{  U \,\Diag(x)\,V^\top  \ : \  U\in O_{\CC}(n), \  V\in O_{\CC}(t), \ x\in Q  \}.
$$
Then $\mathcal{N}_Q$ is a dense subset of $\MM_\CC$, and $\mathcal{N}_Q$ contains an open dense subset of $\MM_\CC$.
\end{theorem}

\begin{proof}
After we show $\mathcal{N}_Q$ is a dense subset of $\MM_\CC$, 
Chevalley's theorem \cite[Theorem 3.16]{harris} will immediately imply that $\mathcal{N}_Q$ contains an open dense subset of $\MM_\CC$, as claimed. 

We first argue the inclusion 
 $\mathcal{N}_{S_\CC}\subseteq \MM_\CC$ (and hence $\mathcal{N}_{Q}\subseteq \MM_\CC$). To this end, 
 for any 
$f\in I(\MM_\CC)$, note that the polynomial $q(x):=f(\Diag(x))$ vanishes on $S$ and therefore on $S_\CC$. 
Hence the inclusion $\{\Diag(x) \ : \  x\in S_\CC\} \subseteq\MM_\CC$ holds. Since $\MM_\CC$ is 
$O_{\CC}(n) \times O_{\CC}(t)$-invariant (Theorem \ref{thm:invariant varieties}), we conclude that $\mathcal{N}_{S_\CC}\subseteq \MM_\CC$, as claimed. Moreover, clearly $\mathcal{M}$ is a subset of $\mathcal{N}_{S_\CC}$, and hence the inclusion
$\MM_\CC\subseteq\overline{\mathcal{N}_{S_\CC}}$ holds. We conclude the equality $\MM_\CC=\overline{\mathcal{N}_{S_\CC}}$.

Now suppose that $Q$ contains an open dense subset of  $S_\CC$ and
consider the continuous polynomial map
$P: O_\CC(n) \times S_\CC \times O_\CC(t) \ra \mathcal{M}_\CC$ 
given by
$$
P(U,x,V) := U \, \Diag(x)\, V^\top .
$$
Noting the equations $\ol{Q}=S_\CC$ and 
$\mathcal{N}_Q = P(O_\CC(n)\times Q \times O_\CC(t))$,
we obtain
\begin{align*}
\ol{\mathcal{N}_Q } & = \ol{P(O_\CC(n)\times Q \times O_\CC(t))}
= \ol{P ( \ol{ O_\CC(n)\times Q \times O_\CC(t)  }  )} \\
& = \ol{ P( O_\CC(n)\times \ol{Q} \times O_\CC(t) ) } = \ol{\mathcal{N}_{S_\CC}} = \MM_\CC.
\end{align*}
Hence $\mathcal{N}_Q$ is a dense subset of $\MM_\CC$, as claimed. 
\end{proof}

\begin{remark}\label{rmk:noalgsvd1} {\rm The variety $\MM_\CC$ may contain matrices that do not admit an algebraic SVD and hence the closure operation in Theorem~\ref{thm:generic points in M_C} is not superfluous.
For example the Zariski closure of $\RR^{2\times 2}_1$ contains the matrix ${\left( {\begin{smallmatrix}1&i\\0&0\end{smallmatrix}}\right)}$, which we saw earlier does not have an algebraic SVD.
}
\end{remark}

Though in the notation of Theorem~\ref{thm:generic points in M_C}, the set $\mathcal{N}_{S_\CC}$ coincides with $\mathcal{M}_{\CC}$ only up to closure, we next show that equality does hold unconditionally when restricted to diagonal matrices.
For any matrix $B\in \CC^{n\times n}$ we define $e_1(B), \ldots, e_n(B)$ to be the $n$ coefficients of the characteristic polynomial of $B$, that is  $e_1(B), \ldots, e_n(B)$ satisfy
$$
\det(\lambda I -B) = \lambda^n - e_1(B)\lambda^{n-1} + \cdots + (-1)^n e_n(B).
$$
For any point $b\in \CC^n$, we define $e_i(b) = e_i (\Diag(b))$ for every $i = 1,\ldots,n$. 
In other words, $e_1(b), \ldots, e_n(b)$ are the elementary symmetric polynomials in $b_1,\ldots, b_n$.

\begin{theorem} \label{thm:SC in terms of MC}
The equality, $S_\CC =\{x\in \CC^n \colon  \Diag(x)\in \MM_\CC\}$, holds.
\end{theorem}
\begin{proof}
The inclusion $\subseteq$ follows immediately from the inclusion $\mathcal{N}_{S_\CC}\subseteq \MM_\CC$ established in Theorem~\ref{thm:generic points in M_C}.
For the reverse inclusion, define the set $$\Omega:=\{y\in\CC^n \ :  \  y_i=e_i(x_1^2,\ldots, x_n^2)\quad\forall x\in S_\CC\}.$$ 
We first claim that $\Omega$ is a variety. To see this, by \cite[Proposition 2.6.4]{sturmfels2008algorithms}, the variety $S_{\CC}$ admits some $\Pi_n^\pm$-invariant defining polynomials $f_1,\ldots,f_k\in \CC[x]$.  Since $f_j$ are invariant under coordinate sign changes, they are in fact symmetric polynomials in the squares $x^2_1,\ldots,x_n^2$. 
 Then by the fundamental theorem of symmetric polynomials, we may write each $f_j$ as some polynomial $q_j$ in the quantities $e_i(x_1^2,\ldots, x_n^2)$. We claim that $\Omega$ is precisely the zero set of $\{q_1,\ldots,q_k\}$. By construction $q_j$ vanish on $\Omega$. Conversely, suppose $q_j(y)=0$ for each $j$.  Letting $x_1^2,\ldots,x^2_n$ be the roots of the polynomial $\lambda^n - y_1\lambda^{n-1} + \cdots + (-1)^n y_n$, we obtain a  point $x\in \CC^n$ satisfying $y_i=e_i(x_1^2,\ldots, x_n^2)$ for each $i$. We deduce then that $x$ lies in $S_{\CC}$ and hence $y$ lies in $\Omega$ as claimed. We conclude that $\Omega$ is closed.

Observe the mapping $\pi\colon \MM_\CC\to\CC^n$ defined by $\pi(X)=(e_1(XX^\top ), \ldots, e_n(XX^\top ))$ satisfies $\pi(\mathcal{N}_{S_\CC})\subseteq \Omega$,
and so we deduce $\pi(\MM_\CC)=\pi(\ol{\mathcal{N}_{S_\CC}})\subseteq \ol{\pi(\mathcal{N}_{S_\CC})}\subseteq \Omega$. Hence for any $y\in\CC^n$ satisfying $\Diag(y)\in  \MM_\CC$, there exists $x\in S_\CC$ satisfying 
$e_i(x^2_1,\ldots,x^2_n)=e_i(y^2_1,\ldots,y^2_n)$ for each index $i=1\ldots,n$. We deduce that $x^2_1,\ldots,x^2_n$ and $y^2_1,\ldots,y^2_n$ are all roots of the same characteristic polynomial of degree $n$. Taking into account that  $S_\CC$ is $\Pi^{\pm}_n$-invariant, we conclude that $y$ lies in $S_\CC$. The result follows.
\end{proof}

We conclude with the following two enlightening corollaries, which in particular characterize matrices in $\mathcal{M}_{\CC}$ admitting an algebraic SVD.

\begin{corollary}[SVD in the closure]  \label{cor:consequences after the theorem of algebraic SVD}
A matrix $X\in \MM_\CC$ admits an algebraic SVD if and only if 
	$XX^\top $ is diagonalizable, $\rank(X) = \rank(XX^\top )$, and the vector of algebraic singular values of $X$ lies in $S_\CC$.
\end{corollary}
\begin{proof}
This follows immediately from Theorems \ref{thm:invariant varieties}, \ref{thm:algebraic SVD}, and \ref{thm:SC in terms of MC}.	
\end{proof}
	
\begin{corollary}[Eigenvalues in the closure]	\label{eigenvalues in the closure}
If $X$ is a matrix in $\MM_\CC$, then the vector of the square roots of the eigenvalues of $XX^\top $ lies in $S_\CC$.  
\end{corollary}
\begin{proof}
Recall that, if $\mathcal{U}$ is an open dense subset of a variety $\VV$, then 
$\mathcal{U}$ has nonempty intersection with any irreducible component of $\VV$; see \cite[1.2 Proposition]{borel2012}.
Hence  the intersection of $\mathcal{U}$  with any irreducible component of $\VV$ is open dense in that component, and is Euclidean dense in that component as well;
see \cite[page 60, Corollary 1]{mumford}. Consequently, $\mathcal{U}$ is Euclidean dense in $\VV$.

From Theorem \ref{thm:generic points in M_C} we know $\mathcal{N}_{S_\CC}$ contains an open dense subset of $\MM_\CC$. 
It follows from the above discussion that $\mathcal{N}_{S_\CC}$ is Euclidean dense in $\MM_\CC$.
Given $X\in \MM_\CC$, we let 
$x$ be the vector of the square roots of the eigenvalues of $XX^\top $, which is defined up to sign and order.
We know there is a sequence $X_k := U_k\, \Diag(x^k)\, V_k^\top $, where $U_k\in O_\CC(n)$, $V_k\in O_\CC(t)$, and 
$x^k\in S_\CC$ such that $X_k \ra X$ as $k\ra\infty$. Hence
$$
(e_1(X_k X_k^\top ), \ldots, e_n(X_k X_k^\top )) \ra (e_1(XX^\top ), \ldots, e_n(XX^\top )).
$$
Since roots of polynomials are continuous with respect to the coefficients \cite[Theorem 1]{zedek}, we deduce that the roots of the characteristic polynomial $\det(\lambda I-X_kX^\top _k)$, namely $((x^k_1)^2, \ldots, (x^k_n)^2)$, converge to $(x_1^2,\ldots,x_n^2)$ up to a coordinate reordering of $x^k$'s and $x$. Passing to a subsequence, we deduce that $x_k$ converge to $x$ up to a signed permutation. Since $S_\CC$ is closed, we conclude that $x$ lies in $S_\CC$, as claimed.
\end{proof}

\subsection{GIT perspective of algebraic SVD} 

The algebraic SVD can be viewed from the perspective of Geometric Invariant Theory (GIT) \cite[Chapter 2]{derksen2013}. 
Let $G$ be the group $O_\CC(n)\times O_\CC(t)$ acting on $\C^{n\times t}$ via  
$(U,V) \cdot A = UAV^\top $. 
For any variety $\VV$ over $\CC$, let $\CC[\VV]$ be the ring of polynomial maps $\VV\ra \CC$. Fix the $G$-invariant variety $\MM_\CC$ and define the {\em invariant ring}
$$
\CC[\MM_\CC]^G:= \{ f\in \CC[\MM_\CC] \ : f \textrm{ is } G\textrm{-invariant}\}
$$
as a subring of $\CC[\MM_\CC]$. 
Consider a map $f\in \CC[\MM_\CC]^G$. Since the map $q(x):=f\circ\Diag(x)$ lies in $\CC[S_\CC]$ and is $\Pi_n^\pm$-invariant,  we may write $q$ as a polynomial map in the values $e_i(x_1^2,\ldots, x_n^2)$. Hence by passing to the limit,
$f$ itself can be expressed as a polynomial over $\CC$ in the ordered sequence of coefficients 
$e_1(XX^\top ) , \ldots, e_n(XX^\top )$. In other words, the following equality holds:
$$
\CC[\MM_\CC]^G = \CC[e_1(XX^\top ) , \ldots, e_n(XX^\top )]
$$
Observe that $\CC[\MM_\CC]^G$ is a finitely generated reduced $\CC$-algebra, and as such, there is a variety over $\CC$ denoted by $\MM_\CC \parallelslant G$, such that 
$\CC[\MM_\CC]^G$ is isomorphic to $\CC[\MM_\CC \parallelslant G]$. This variety (up to isomorphism) is called the 
{\em GIT quotient}, and is denoted by $\MM_\CC\parallelslant G$. Concretely, we may write
$\MM_\CC\parallelslant G$ as the variety corresponding to the ideal 
$$\{f\in \CC[x]: f(e_1(XX^\top ),\ldots,e_n(XX^\top ))=0\quad \textrm{ for all }X\in \MM_{\CC}\}.$$
A bit of thought shows that in our case, we may equivalently write
$$\MM_\CC\parallelslant G=\{y\in\CC^n \ :  \  y_i=e_i(x_1^2,\ldots, x_n^2)\quad\forall x\in S_\CC\}$$
This was already implicitly shown in the proof of Theorem~\ref{thm:SC in terms of MC}.

 The {\em quotient map} $\pi: \MM_\CC \ra \MM_\CC\parallelslant G$ is the surjective polynomial map associated 
to the inclusion $\CC[\MM_\CC]^G \hookrightarrow \CC[\MM_\CC]$. 
To be precise, in our case we have 
$$
\pi( X) = (e_1(XX^\top ) , \ldots, e_n(XX^\top ))
$$
Intuitively  $\MM_\CC \parallelslant  G$ can be ``identified'' with the space of closed orbits for the action of $G$ on $\MM_\CC$, but not the orbit space. 
It can be proved that a $G$-orbit in $\MM_\C$ is closed if and only if it is the orbit of a diagonal matrix. 
In other words, the orbit of a matrix $X$ is closed if and only if $X$ admits an algebraic SVD. By contrast, 
all $O(n)\times O(t)$-orbits in $\MM$ are closed (compare these facts with \cite[\S 16]{procesi}).

\section{ED critical points of an orthogonally invariant variety}\label{sec:complexED}

We are now ready to prove our main results characterizing ED critical points of a data point $Y \in \CC^{n \times t}$ with respect to an orthogonally invariant matrix variety $\MM \subseteq \RR^{n \times t}$. 
We first  give the precise definition of an ED critical point; see
\cite[\S 2]{DHOST}.
For any variety $\VV$ over $\RR$ or $\CC$, we let $\VV^{\rm reg}$ be the open dense subset of regular points in $\VV$. Recall that if $\mathcal{V}$ is a union of irreducible varieties $\mathcal{V}_i$, then $\VV^{\rm reg}$ is the union of $\VV^{\rm reg}_i$ minus the points in the intersection of any two irreducible components.
In what follows, for any two vectors $v,w\in\CC^n$, the symbol $v\perp w$ means $v^\top  w =0$, and for any set $Q\subseteq\CC^n$ we define $Q^{\perp}:=\{v\in\CC^n: v^\top  w=0 \textrm{ for all } w\in Q\}$ .

\begin{definition}[ED critical point, ED degree]\label{defn:edcrit} 
{\rm Let $\VV$ be a real variety in $\RR^n$ and consider a data point $y\in \CC^n$. An {\em ED critical point} of $y$ with respect to $\VV$
is a 
point $x\in \VV_\CC^{\textup{reg}}$ such that  
$y-x\in \mathcal{T}_{\VV_\CC}(x)^{\perp}$, where $\mathcal{T}_{\VV_\CC}(x)$ is the tangent space of $\VV_\CC$ at $x$. 

For any generic point $y$ in $\CC^n$, the number of ED critical points of $y$ with respect to $\VV$ is a  constant; see \cite{DHOST} called the 
{\em EDdegree} of $\VV$ and denoted by ${\rm EDdegree}(\VV)$. 
} 
\end{definition}

Here is a basic fact that will be needed later. 

\begin{lemma} \label{lem:generic data point gives generic critical points} Let $\VV \subseteq \RR^n$ be a variety and let
$\mathcal{W}$ be an open dense subset of $\VV_\CC$. Then all ED critical points of a generic $y \in \CC^n$ with respect to $\VV$ lie in $\mathcal{W}$. 
\end{lemma}

\begin{proof}
The proof is a dimensional argument explained in \cite{DHOST}. 
Without loss of generality assume that $\VV_\CC$ is irreducible. 
Consider the ED correspondence ${\mathcal E}_{\VV_\CC}$,
as defined in \cite[\S 4]{DHOST},
with its two projections $\pi_1$ on $\VV_\CC$ and $\pi_2$ on $\CC^{n\times t}$.
Since $\pi_1$ is an affine vector bundle over $\VV_\CC^{reg}$, it follows that
 $\pi_2\left(\pi_1^{-1}(\VV_\CC\setminus\mathcal{W})\right)$ has dimension smaller than $nt$.
\end{proof}

	\begin{remark}
			{\rm
				We mention in passing, that the ED degree of a variety $\VV$, as defined above equals the sum of the ED degrees of its irreducible components $\VV_i$, which coincides with the original definition of ED degree in \cite{DHOST}. This follows from Lemma~\ref{lem:generic data point gives generic critical points} by noting that the
				set  $\VV^{\rm reg}_{\CC}\cap (\VV_i)_{\CC}$ is an open dense subset of $(\VV_i)_{\CC}$ for each $i$. 
				}
	\end{remark}	

We say that two matrices $X$ and $Y$ admit a {\em simultaneous algebraic SVD} if there exist orthogonal matrices $U\in O_{\CC}(n)$ and $V \in O_{\CC}(t)$ so that both $U^\top XV$ and $U^\top YV$ are diagonal matrices.
Our first main result is that every ED critical point $X$ of a generic matrix $Y \in \CC^{n \times t}$ with respect to an orthogonally invariant variety $\mathcal{M}$ admits a simultaneous algebraic SVD with $Y$. 

\begin{theorem}[Simultaneous SVD] \label{thm:diagonal critical point}
Fix an $O(n)\times O(t)$-invariant matrix variety $\mathcal{M}\subseteq\RR^{n\times t}$. 
Consider a matrix $Y\in\CC^{n\times t}$ so that the eigenvalues of $YY^\top $ are nonzero and distinct.
Then any ED critical point $X$ of $Y$ with respect to $\MM$  admits a simultaneous algebraic SVD with $Y$.
\end{theorem}

The proof of this theorem relies on the following three  lemmas.

\begin{lemma} \label{lem:tangent space of On}
The tangent space of $O_{\CC}(n)$ at a point $U \in O_{\CC}(n)$ is 
\begin{align*}
\TT_{O_{\CC}(n)}(U) & = \{ZU \ :  \  \text{$Z\in \CC^{n\times n}$ is skew-symmetric}\}\\
&= \{UZ \ :  \  \text{$Z\in \CC^{n\times n}$ is skew-symmetric}\}.
\end{align*}
\end{lemma}

\begin{proof}
Recall $O_{\CC}(n) = \{ W \in \CC^{n \times n} \,:\, WW^\top  = I \}$.
Consider the map $F: \CC^{n\times n}\ra \CC^{n\times n}$ given by $W\mapsto WW^\top $. 
Note that for any $W,B\in \CC^{n\times n}$ and $t\in \RR$, one has 
$$
(W+tB)(W+tB)^\top  = WW^\top  + t(WB^\top  + BW^\top ) + t^2 BB^\top .
$$
Hence given  $U\in O_{\CC}(n)$, we have
$[\nabla F(U)](B) = UB^\top  + BU^\top $.
The tangent space $\TT_{O_\CC(n)}(U)$ is the kernel of the linear map $\nabla F(U)$.
Consider the matrix $Z:=BU^\top $. 
Then $[\nabla F(U)](B) = 0$ if and only if 
$Z^\top  + Z = 0$ which means $Z$ is skew-symmetric. This proves the first description of $\TT_{O_\CC(n)}(U)$.
The second description follows by considering the map $W\mapsto W^\top  W$ instead of $F$.
\end{proof}

\begin{lemma} \label{lem:symmetric}
A matrix $A\in \CC^{n\times n}$ is symmetric if and only if 
${\rm trace}(AZ)=0$ for any skew-symmetric matrix $Z\in \CC^{n\times n}$.
\end{lemma}

\begin{proof}
The ``if" part follows because $A_{ij} - A_{ji} = {\rm trace}(A (E^{ij} - E^{ji}))$ 
where $E^{ij}$ denotes the $n\times n$ matrix whose $(i,j)$-entry is one and all other entries are zero. 
The ``only if" part follows by the same reasoning since $\{E^{ij} - E^{ji}\}$ is a basis for the 
space of skew-symmetric matrices.
\end{proof}

\begin{lemma} \label{lem:diagonal}
Consider a matrix $A\in \CC^{n\times t}$ and a diagonal matrix
$D\in \CC^{n\times t}$ with nonzero diagonal entries $d_i$ such that the squares $d_i^2$ are distinct.
Then if $AD^\top $ and $D^\top  A$ are both symmetric, the matrix $A$ must be diagonal.
\end{lemma}

\begin{proof}
The symmetry of $AD^\top $ means $A_{ij}d_j = A_{ji}d_i$ for any $i,j=1,\ldots,n$.
In addition, the symmetry of $D^\top  A$ implies  
$A_{ij}d_i = A_{ji} d_j$ for all $i,j=1,\ldots,n$ and 
$A_{ij}d_i = 0$ for any $i=1,\ldots,n$ and $j>n$.
Therefore for any $i,j$, one has 
$$
A_{ij}d_id_j = A_{ji} d_i^2 \qquad \textrm{ and }\qquad A_{ij}d_id_j= A_{ji} d_j^2.
$$
Since $d_i^2\neq d_j^2$ for all $i\neq j$, we get 
$A_{ij}=0$ for all $i \neq j$, $i,j=1,\ldots,n$. 
Since the $d_i$'s are all nonzero and $A_{ij}d_i = 0$ for any $i=1,\ldots,n$ and $j>n$, 
we have $A_{ij} = 0$ for any $i = 1,\ldots,n$ and $j>n$. Thus $A$ is diagonal.
\end{proof}

\begin{remark}{\rm  The assumption $d_i^2\neq d_j^2$ for $i\neq j$ is necessary in Lemma \ref{lem:diagonal}. For example consider $D=I$
and the symmetric matrices
$$A=\begin{pmatrix}\cos\theta&-\sin\theta\\
-\sin\theta&-\cos\theta\end{pmatrix}\in O(2), \,\,\,\theta\in \RR $$
for which 
$AD^\top $ and $D^\top  A$ are both symmetric. However, $A$ is diagonal only when $\theta=k\pi$ with $k\in{\mathbb Z}$.}
\end{remark}

\noindent{\em Proof of Theorem~\ref{thm:diagonal critical point}.}
By Corollary~\ref{cor:algebraicSVD as we had}, we may write $Y = UDV^\top $  for some $U\in O_\CC(n)$, $V\in O_\CC(t)$, and a diagonal matrix $D\in \C^{n\times t}$.
Let $X$ be an ED critical point of $Y$ with respect to $\MM$. 
Then $A:= U^\top  X V$ lies in $\MM_\CC$ (Theorem~\ref{thm:invariant varieties}).
To prove the theorem, we  need to show that $A\in \CC^{n\times t}$ is diagonal. 

Consider the map $F: O_{\CC}(n) \ra \mathcal{M}_\CC$ given by $W\mapsto WAV^\top $. 
Then 
$$
[\nabla F(U)](B) = BAV^\top  \in \TT_{\mathcal{M}_\CC}(X),
$$
for any $B\in \TT_{O_{\CC}(n)}(U)$.
By Lemma \ref{lem:tangent space of On}, we may write $B = UZ$ for a skew-symmetric $Z$, yielding
$UZAV^\top  \in \TT_{\mathcal{M}_\CC}(X)$. Varying $B$, we see that 
the tangent space of $\mathcal{M}_\CC$ at $X$ contains $\{UZAV^\top  \ : \  Z^\top  = -Z\}$. 
Then, by the definition of ED critical point we have ${\rm trace}((Y-X) (UZAV^\top )^\top )= 0$ 
for any skew-symmetric matrix $Z$, and hence
$$
0 = {\rm trace}(U(D-A)V^\top  VA^\top  Z^\top  U^\top ) = {\rm trace}((D-A)A^\top  Z^\top ). 
$$
 By Lemma \ref{lem:symmetric}, this means $(D-A)A^\top $ is symmetric.
Since $AA^\top $ is symmetric, we have that $DA^\top $ is symmetric; therefore the transpose $AD^\top $ is symmetric.

By considering $F \,:\, O_{\CC}(t) \ra \mathcal{M}_\CC$ given by $W \mapsto UAW^\top $, we get as above, that 
$\{UAZ^\top  V^\top  \ : \  Z^\top  = -Z\} \subseteq \TT_{\MM_\CC}(X)$.  It follows that  
$$
0 = {\rm trace}((U(D-A)V^\top )^\top  UA Z^\top  V^\top ) = {\rm trace}((D-A)^\top  A Z^\top )
$$
for any skew-symmetric matrix $Z$, and by Lemma \ref{lem:symmetric}, $(D-A)^\top  A$ is symmetric.
Again, since $A^\top  A$ is symmetric, we get that $D^\top  A$ is symmetric.
Since $AD^\top $ and $D^\top  A$ are both symmetric, we conclude $A$ is diagonal  by Lemma \ref{lem:diagonal}, as claimed.
\qed

\medskip
The next ingredient in our development is a version of Sard's Theorem in algebraic geometry (often 
called ``generic smoothness" in textbooks); see  \cite[III, Corollary 10.7]{hartshorne1977algebraic-geometry}. 

\begin{lemma}[Generic smoothness on the target] \label{sard theorem}
Let $\VV$ and $\mathcal{W}$ be  varieties over $\CC$. 
Consider a polynomial map $f: \VV\ra \mathcal{W}$. Then  there is an open dense subset $\mathcal{W}'$ of 
$\mathcal{W}^{\textup{reg}}$ (and hence of $\mathcal{W}$) such that for any $w\in \mathcal{W}'$ and any point $v\in \mathcal{V}^{\textup{reg}}\cap f^{-1}(w)$, the linear map 
$\nabla f(x)\colon \mathcal{T}_{\VV}(x)\to \mathcal{T}_{\mathcal{W}}(f(x))$ is surjective. 
\end{lemma}

We now establish a key technical result: a representation of the tangent space of $\MM_\CC$ 
at a generic matrix $X\in \MM_\CC$ in terms of the tangent space of $S_\CC$ at the vector of algebraic singular values of $X$. 

\begin{lemma}[Transfer of tangent spaces] \label{transfer of tangent spaces}
Consider a $\Pi^{\pm}_n$-invariant variety $S \subseteq \RR^n$ and the induced real variety $\MM := \sigma^{-1}(S)$.
Then the following statements hold.
\begin{enumerate}[label=({\alph*}), leftmargin=0.7cm]
\item \label{transfer of tangent spaces: part 1}
A generic point $X\in \MM_\CC$ lies in $\MM_\CC^{\textup{reg}}$, 
admits an algebraic SVD, and its vector of algebraic singular values lies in $S_\CC^{\rm reg}$.
Moreover, the tangent space $\TT_{\MM_\CC}(X)$ admits the representation
\begin{equation} \label{eqn:tangent space}
\TT_{\MM_\CC}(X)= \lt\{\begin{array}{c}
UZ_1\, \Diag(x)\, V^\top +U\, \Diag(x)\, Z_2^\top  V^\top  + U\, \Diag(a)\, V^\top  \ : \  \\
 a\in \mathcal{T}_{S_\CC}(x), \  Z_1, \ Z_2 \textrm{ are skew-symmetric}
\end{array}
 \rt\},
\end{equation}
 for any $U\in O_\CC(n)$, $V\in O_\CC(t)$, and $x\in S_\CC^{\rm reg}$ satisfying $X = U \, \Diag(x) \, V^\top $.

\item\label{it:claim_gen}
A generic point $x\in S_\CC$ lies in $S_\CC^{\textup{reg}}$. Moreover, for any $U\in O_\CC(n), V\in O_\CC(t)$, 
the point $X= U\,\Diag(x)\, V^\top $ lies in $\MM_\CC^{\textup{reg}}$, and satisfies \eqref{eqn:tangent space}.
\end{enumerate}
\end{lemma}

\begin{proof}
We begin by proving claim $\ref{transfer of tangent spaces: part 1}$.
By Theorem~\ref{thm:generic points in M_C} with $Q = S_\CC^{\rm reg}$,
a generic point $X\in \MM_\CC$ 
admits an algebraic SVD: $X = U' \, \Diag(x') \, V'^\top $ for some $x'\in S_\CC^{\rm reg}$. 
As $\mathcal{M}_\CC^{\rm reg}$ is an open dense subset of $\mathcal{M}_\CC$, we can assume that $X$ lies in $\mathcal{M}_\CC^{\rm reg}$. 
Consider the  polynomial map
$P: O_\CC(n) \times S_\CC \times O_\CC(t) \ra \mathcal{M}_\CC$ 
given by
$$
P(\wt{U},\wt{x},\wt{V}) := \wt{U} \, \Diag(\wt{x})\, \wt{V}^\top .
$$
By Lemma~\ref{sard theorem} we can assume that $\nabla P(U,x,V)$ is surjective 
whenever we can write $X=U \, \Diag(x) \,V^\top $ for some 
$U\in O_\CC(n)$, $V\in O_\CC(t)$ and $x\in S_\CC^{\rm reg}$. 
Therefore the description of tangent space in 
\eqref{eqn:tangent space} follows from Leibniz rule on $P$ and Lemma~\ref{lem:tangent space of On}. Hence claim \ref{transfer of tangent spaces: part 1} is proved. 

	Next, we argue claim \ref{it:claim_gen}. To this end, let $\Theta$ be the dense open subset of $\MM_{\CC}$ guaranteed to exist by \ref{transfer of tangent spaces: part 1}. We claim that we can assume that $\Theta$ is in addition orthogonally invariant. To see this, observe that all the claimed properties in \ref{transfer of tangent spaces: part 1} continue to hold on the dense, orthogonally invariant subset 
	$\Gamma:=\bigcup\{U\Theta V^T: U\in O_{\CC}(n),\, V\in  O_{\CC}(t)\}$ of $\MM_{\CC}$. 
	By Lemma~\ref{lem:closure is Ginvariant}, the set $\overline{\mathcal{M}_{\CC}\setminus\Gamma}$ is an orthogonally invariant variety. 
	Note now the inclusions  
	$\Theta\subseteq \mathcal{M}_{\CC}\setminus (\overline{\mathcal{M}_{\CC}\setminus\Gamma})\subseteq\Gamma$.
		It follows that  $\mathcal{M}_{\CC}\setminus (\overline{\mathcal{M}_{\CC}\setminus\Gamma})$ is
		an orthogonally invariant, open, dense variety in $\MM_{\CC}$
on which all the properties in \ref{transfer of tangent spaces: part 1} hold. Replacing $\Theta$ with $\mathcal{M}_{\CC}\setminus (\overline{\mathcal{M}_{\CC}\setminus\Gamma})$, we may assume that $\Theta$ is indeed orthogonally invariant in the first place.

	Next, we appeal to some results of Section~\ref{subsec:components}.
	Let $\{S_i\}^k_{i=1}$ be the irreducible components of $S$ and define the symmetrizations $S^{\pi}_i:=\bigcup_{\pi \in \Pi^{\pm}_n} \pi S_i$ and the varieties $\mathcal{M}_i:=\sigma^{-1}(S^{\pi}_i)$.
	 Observe that $\overline{S_i}$ are the irreducible components of $S_{\CC}$ and we have $\overline{S^{\pi}_i}=\bigcup_{\pi \in \Pi^{\pm}_n} \pi \overline{S_i}$. Note also that $\mathcal{M}_{\CC}$ is the union of the varieties $\overline{\mathcal{M}_{i}}$.
	 	 \smallskip
	 	 
	 By Theorem~\ref{thm:union_comp}, 	 
 each variety $\overline{\MM_i}$ is a union of some irreducible components of $\MM_{\CC}$.
	Since the intersection of $\mathcal{M}^{\textup{reg}}_{\CC}$ with any irreducible component of  $\mathcal{M}_{\CC}$ is open and dense in that component, we deduce that the intersection $\overline{\mathcal{M}_{i}}\cap\mathcal{M}^{\textup{reg}}_{\CC}$ is an open dense subset of $\overline{\mathcal{M}_{i}}$ for each index $i$.
 Similarly, the intersection $\Theta\cap \overline{\mathcal{M}_i}$ is open and dense in each variety $\overline{\mathcal{M}_i}$.
	Then clearly $\Theta$ intersects $\mathcal{N}_{\overline{S^{\pi}_{i}}}$ for each index  $i$, since by Theorem~\ref{thm:generic points in M_C} the set $\mathcal{N}_{\overline{S^{\pi}_{i}}}$ contains an open dense subset of $\overline{\mathcal{M}_i}$.
	Therefore for each index $i$, the set $\Theta$ contains $\Diag(x_i)$ for some $x_i\in \overline{S^{\pi}_{i}}$. 
	
	We deduce that the diagonal restriction of $\Theta$, namely the set 
	$$W:=\{x\in \CC^{n}: \Diag(x)\in\Theta\},$$
	is an absolutely symmetric, open subset of $S_{\CC}$ and it intersects  each variety  $\overline{S^{\pi}_{i}}$. In particular, $W$ intersects each 
	 irreducible component $\overline{{S}_{i}}$. Since nonempty open subsets of irreducible varieties are dense, we deduce that $W$ is dense in $S_{\CC}$.
		 Moreover, since for any point $x\in W$, the matrix $\Diag(x)$ lies in $\Theta$, we conclude 
	\begin{itemize}[leftmargin=0.7cm]
		\item $\Diag(x)$ lies in $\mathcal{M}_{\CC}^{\textup{reg}}$ (and hence by orthogonal invariance so do all matrices  $U\,\Diag(x)\, V^\top $ with $U\in O_{\CC}(n), V\in O_{\CC}(t)$) and $x$ lies in $S_{\CC}^{reg}$, 
		\item equation \eqref{eqn:tangent space} holds for $X=\Diag(x)$, 
		 and hence by orthogonal invariance of $\mathcal{M}_{\CC}$ and of the description \eqref{eqn:tangent space}, the equation continues to hold for  $X=U\,\Diag(x)\,V^\top $, where $U$ and $V$ arbitrary orthogonal matrices.
	\end{itemize}
	Thus all the desired conclusions hold for any $x$ in the open dense subset $W$ of $S_{\CC}$.
The result follows.
\end{proof}

We are now ready to prove the main result of this paper,  equation \eqref{ED_res} from the introduction. As a byproduct, we will establish an explicit bijection between the ED critical points of a generic matrix $Y = U \, \Diag(y)\, V^\top  \in \CC^{n\times t}$ 
on $\MM$ and the ED critical points of $y$ on $S_{\CC}$.

\begin{theorem}[ED degree]\label{thm:main} Consider a $\Pi^{\pm}_n$-invariant variety $S \subseteq \RR^n$ and the induced real variety $\MM := \sigma^{-1}(S)$.
Then a generic matrix $Y\in\CC^{n\times t}$ admits a decomposition $Y=U \, \Diag(y) \, V^\top $, for some matrices $U \in O_{\CC}(n)$, $V \in O_{\CC}(t)$, and  $y\in \CC^n$. Moreover, then the set of ED critical points of $Y$ with respect to $\MM$ is 
$$
\{U \, \Diag(x) \, V^\top  \ : \   \text{$x$ is an ED critical point of $y$ with respect to $S$}\},
$$
In particular, equality $\mathrm{EDdegree}(\MM) = \mathrm{EDdegree}(S)$ holds.
\end{theorem}

\begin{proof} 
For generic $Y\in \CC^{n\times t}$, the eigenvalues of $YY^\top $ are nonzero and distinct.
Then by Corollary~\ref{cor:algebraicSVD as we had}, we can be sure that $Y$ admits an algebraic SVD. We fix such a decomposition $Y=U \, \Diag(y) \, V^\top $, for some $U \in O_{\CC}(n)$, $V \in O_{\CC}(t)$, and  $y\in \CC^n$.

Let $X$ be an ED critical point of $Y$ with respect to $\mathcal{M}$. By Theorem~\ref{thm:diagonal critical point}, we can assume that $X$ and $Y$ admit a simultaneous SVD, that is
both $U'^\top XV'^\top $ and $U'^\top YV'^\top $ are diagonal for some matrices $U' \in O_{\CC}(n)$, $V' \in O_{\CC}(t)$. Notice that the columns of $U$ and $U'$ are equal up to a sign change and a permutation. Similarly the first $n$ columns of $V$ and $V'$ are equal up to a sign change and a permutation. Hence we may assume that $X$ can be written as $X=U\,\Diag(x)\,V^\top $ for some $x\in S_{\CC}$.
By  Lemmas~\ref{lem:generic data point gives generic critical points} and \ref{transfer of tangent spaces}, we can further assume that $X$ lies in $\MM_\CC^{\rm reg}$ and $x$ lies in $S_\CC^{\rm reg}$, and moreover the
 tangent space $\TT_{\MM_\CC}(X)$ at $X = U \, \Diag(x) \,V^\top $ is given in \eqref{eqn:tangent space}. 

We will now show that $x$ is an ED critical point of $y$ with respect to $S$. To see this, observe the inclusion
$$
\{U \, \Diag(a) \, V^\top  \ : \ a\in \mathcal{T}_{S_\CC}(x)\} \subseteq \mathcal{T}_{\mathcal{M}_\CC}(X).
$$
and hence $$0={\rm trace}(U\,\Diag(y-x)\, V^\top (U\, \Diag(a)\, V^\top )^\top ) \quad\textrm{for any } a\in  \mathcal{T}_{S_\CC}(x).$$
Simplifying, we immediately conclude
$(y-x)^\top  a=0$ for any $a\in  \mathcal{T}_{S_\CC}(x)$, and hence $x$ is an ED critical point of $y$ with respect to $S$.

Conversely, suppose $x\in S^{\textup reg}_{\CC}$ is an ED critical point of $y$ with respect to $S$. Applying Theorem~\ref{thm:generic points in M_C} with $Q=\CC^n$, we deduce that if a set $Q\subseteq \CC^{n}$ contains an open dense set in $\CC^n$, then $$\{\widehat U \,\Diag(z)\,\widehat{V}^\top : z\in Q,~ \widehat U\in O_{\CC}(n),~\widehat V\in O_{\CC}(t)\}$$
contains an open dense subset of $\CC^{n\times t}$.
Define now the matrix $X:=U\, \Diag(x)\, V^\top $.
Then by Lemmas \ref{lem:generic data point gives generic critical points} and \ref{transfer of tangent spaces}, we may assume that 
$X$ is regular and the
tangent space of $\MM_\CC$ at $X$ is generated
by all matrices of the form 
\begin{enumerate} [label=\emph{\roman*})]
\item\label{it1}
$UZ \, \Diag(x) \,V^\top $ with $Z$ skew-symmetric,
\item\label{it2}
$U \, \Diag(a) \, V^\top $ where $a$ belongs to the tangent space of $S_\CC$ at $x$,
\item\label{it3}
$U \, \Diag(x) \, Z^\top  V^\top $ with $Z$ skew-symmetric.
\end{enumerate}
We will show \begin{equation}\label{eq:criticalX}
Y-X\perp \TT_{\MM_\CC}(X)
\end{equation}
by dividing the proof according to the three cases \ref{it1},\ref{it2}, \ref{it3} above.
For \ref{it1}, observe 
$$\mathrm{trace}\left((X-Y)(UZ \, \Diag(x) \, V^\top )^\top \right) =
\mathrm{trace}\left(\Diag(x-y) \, \Diag(x)^\top  Z^\top  \right ) =0,$$
where the last equality follows from Lemma \ref{lem:symmetric}. The 
computation for  \ref{it3} is entirely analogous. 
For  \ref{it2}, we obtain 
$$\mathrm{trace}\Big((X-Y) (U \, \Diag(a) \, V^\top )^\top \Big) = \mathrm{trace}\Big((\Diag(x-y) \Diag(a)^\top \Big)=0,$$ 
where the last equation follows from the hypothesis that $x$ is an ED critical point of $y$ on $S_\CC$. 
We conclude that $X$ is an ED critical point of $Y$ relative to $\MM_{\CC}$, as claimed. The equality, $\mathrm{EDdegree}(\MM) = \mathrm{EDdegree}(S)$,  quickly follows.
\end{proof}

\begin{example} \label{ex:4 examples}
{\rm
To illustrate Theorem~\ref{thm:main}, we now derive the ED degree of some notable orthogonally invariant varieties  summarized in the following table. The pairs $(\MM, S)$ in all these examples were also discussed in \cite[Section~4]{DLT}. The dimension of $\MM$ or $\MM_\CC$ can be computed using \eqref{eq:dim formula}.\\

\begin{center}
\begin{tabular}{|cc|cc|c|}
\hline
$\substack{\text{orthogonally}\\\text{invariant}\\\text{variety }\MM}$ & dimension
 & $\substack{\text{absolutely}\\\text{symmetric}\\\text{variety }S}$  & dimension  & EDdegree\\
\hline
$\RR^{n\times t}_r$ & $r(n+t-r)$ & $\RR^n_r$ & $r$ & ${n \choose r}$ \\
\hline
$\mathcal{E}$ & 6 & $E_{3,2}$ & 1 & 6\\
\hline
$O(n)$ & ${n \choose 2}$ & $\{(\pm 1, \ldots,\pm 1)\}$ & 0 & $2^n$\\
\hline
$SL_n^\pm$ & $n^2-1$ & $H_n$ & $n-1$ & $n 2^n$\\
\hline
$\mathcal{F}_{n,t,d}$ ($d$ even) & $nt-1$  & $F_{n,d}$ & $n-1$ & \cite[Cor. 2.12]{hwangrae} \\
\hline
\end{tabular}
\end{center}

\vspace{0.2cm}

In the first three examples, the set $S$ is a subspace arrangement and hence its ED degree is the number of distinct maximal subspaces in the arrangement. We will elaborate on this situation in Section~\ref{subsec:subspaceArrangements}.

The matrix variety $SL_n^\pm$ consists of all matrices $A \in \CC^{n \times n}$ satisfying $\det(A) = \pm 1$. 
The ED degree of $\textup{SL}_n^{\pm}$ was explicitly computed in \cite{BaaijensDraisma}. We show below how our main theorem provides a simple alternate proof of their result.

The absolutely symmetric variety  
$S$ in this case is $H_n := \{ x \in \RR^n \,:\, x_1 x_2 \cdots x_n = \pm 1 \}$. 
To compute the ED degree of $H_n$, we add up the ED degrees of its two irreducible components
$$
H_n^+ := \{ x \in \RR^n \,:\, x_1 x_2 \cdots x_n =  1 \}
$$
and
$$
H_n^- := \{ x \in \RR^n \,:\, x_1 x_2 \cdots x_n =  -1 \}.
$$
To compute the ED degree of $H_n^+$, we begin with a point $y\in \CC^n$. Then by a straightforward computation, 
$x$ is an ED critical point of $y$ with respect to $H_n^+$ if and only if $x$ solves the system
\begin{align} \label{eqn:KKT for Hn+}
\begin{cases}
x_i (x_i-y_i) = x_n (x_n-y_n) \quad\text{ for all $i= 1,\ldots,n-1$}\\
 x_1 \cdots x_n = 1.
\end{cases}
\end{align}
By B\'{e}zout's Theorem, 
we know EDdegree$(H_n^+) \leq  n2^{n-1}$. We now argue that the data point $y=0$ has 
$n2^{n-1}$ ED critical points with respect to $H_n^+$ which proves that EDdegree$(H_n^+) =  n2^{n-1}$.

When $y=0$, the system \eqref{eqn:KKT for Hn+} is equivalent to 
\begin{align*}
\begin{cases}
x_1^2 = \cdots = x_{n}^2\\
 x_1 \cdots x_n = 1.
\end{cases}
\end{align*}
which has $n2^{n-1}$ solutions in $(H_n^+)_\CC$. Indeed, choose $x_1$ such that  $x_1^{n}=\pm 1 \  \text{($2n$ choices)}$; then 
choose $x_i$ for $i=2,\ldots, n-1$ such that $x_i^2 = x_1^2$ \text{($2$ choices for each $i$)}; finally set $x_n = \f{1}{x_1 \cdots x_{n-1}}$.
Hence EDdegree$(H_n^+) = n2^{n-1}$. Similarly, EDdegree$(H_n^-) = n2^{n-1}$, and therefore we conclude EDdegree$(H_n)=n2^n$.

The variety $\mathcal{F}_{n,t,d} = \{ X \in \RR^{n \times t} \,:\, \| X \|_d = 1 \}$  is the unit ball of the 
Schatten $d$-norm $\|X\|_d := \left[ \sum_{i=1}^n \sigma_i(X)^d \right]^{\frac{1}{d}}$. When $d$ is even, 
the corresponding absolute symmetric variety is the 
{\em affine Fermat hypersurface} 
\begin{eqnarray*} 
F_{n,d} := \lt\{x\in \RR^n \ : \  \sum_{i=1}^n x_i^d = 1\rt\}.
\end{eqnarray*}
The ED degree of a Fermat hypersurface was computed in \cite{hwangrae}.
}
\end{example}


\section{Orthogonally invariant varieties from subspace arrangements}\label{subsec:subspaceArrangements}
In this section, we augment the results of the previous section in the special (and important) case when $S$ is a subspace arrangement. 
Many important matrix varieties, such as the rank varieties $\RR^{n \times t}_r$ and the essential variety  $\EE$, fall in this category. 
Recall that $S$ is a {\em subspace arrangement} if $S$ can be written as a union of finitely many affine subspaces $\{S_i\}^k_{i=1}$ of $\RR^n$. Assuming that the representation of $S$ is chosen in such a way that $S_i$ is not contained in $S_j$ for any distinct $i,j$, we call $S_i$ the {\em affine components} of $S$.
The following result follows directly from Theorem~\ref{thm:main}. 

\begin{corollary}[Affine arrangements] \label{cor:subspace arrangement}
Consider a $\Pi^{\pm}_n$-invariant subspace arrangement $S \subseteq \RR^n$ with affine components $\{S_i\}^k_{i=1}$, and define the induced real variety $\MM := \sigma^{-1}(S)$. Then the equality, $\textup{EDdegree}(\MM) = k$, holds. 

Moreover, a general data point $Y$ in $\RR^{n\times t}$ has  
 exactly $k$ ED-critical points with respect to $\mathcal{M}$: for any decomposition  $Y=U\, \Diag(\sigma(Y)) \, V^\top $ with orthogonal matrices $U\in O(n)$ and $V\in O(t)$, the set of ED-critical points is precisely
$$\{ U\, \Diag(x) \, V^\top : x\textrm{ is the orthogonal projection of } \sigma(Y) \textrm{ onto } S_i\}.$$
In particular, all ED critical points of $Y$ with respect to $\MM$ are real.
\end{corollary}

\begin{proof} 
	Let $\Theta$ be the dense open subset of $\CC^{n\times t}$ guaranteed to exist by Theorem~\ref{thm:main}. Clearly we can also assume that each matrix $Y\in \Theta$ has $\textup{EDdegree}(\MM)$ many ED critical points with respect to $\mathcal{M}$.
	A standard argument shows that the set 
	$\Theta_{\RR}:=\{Y\in \RR^{n\times t}: Y\in \Theta\}$
 	is a dense open subset of $\RR^{n\times t}$. Fix a matrix $Y\in \Theta_{\RR}$ and consider a singular value decomposition $Y=U\, \Diag(\sigma(Y)) \, V^\top $ with orthogonal matrices $U\in O(n)$ and $V\in O(t)$. 
 	By Theorem~\ref{thm:main}, the set of ED critical points of $Y$ with respect to $\MM$ is given by
 	$$
 	\{U \, \Diag(x) \, V^\top  \ : \   \text{$x$ is an ED critical point of $\sigma(Y)$ with respect to $S$}\}.
 	$$
 	Since $\sigma(Y)$ is a real vector, the ED critical points of $\sigma(Y)$ with respect to $S$ are precisely the orthogonal projections of $\sigma(Y)$ on each component $S_i$. Therefore we deduce $k=\textup{EDdegree}(S)=\textup{EDdegree}(\MM)$. 
\end{proof}

The first three examples in Example~\ref{ex:4 examples} illustrate Corollary~\ref{cor:subspace arrangement}. 
Typically, as the data point $y \in \RR^n$ varies, the number of real ED critical points of $y$ with respect to a  variety $\VV \subseteq \RR^n$ varies. Corollary~\ref{cor:subspace arrangement} shows that when $S$ is a subspace arrangement, 
all ED critical points of a real data point with respect to $\MM = \sigma^{-1}(S)$ are again real and their number is constant. This unusual feature is easy to see using Theorem~\ref{thm:main} that creates a bijection between the ED critical points of $\MM$ and $S$, but is not at all obvious if $S$ is not in the picture. 

In common examples, outside of the subspace arrangement case, all ED critical points of a real data point may be purely imaginary and the number of real critical points typically varies as the data point moves around. For instance, the hyperbola $H_n$ in Example~\ref{ex:4 examples} can have complex ED critical points for a generic $y \in \RR^n$. The same is therefore true for  $\textup{SL}_n^{\pm}$.

In a sense, Corollary~\ref{cor:subspace arrangement} generalizes the fact that 
the pairs of singular vectors of a real matrix are real. 
Indeed, the pairs of singular vectors of a real matrix $Y$ correspond to the ED critical points of $Y$ with respect to the 
orthogonally invariant variety of rank one matrices; the corresponding absolutely symmetric variety is the union of 
all coordinate axes.

\begin{remark}\label{rem:symmetric} {\rm Results analogous to those in this paper hold for symmetric matrices under the action of the orthogonal group 
$U \cdot A = U A U^\top $. More precisely, consider the space of real $n\times n$ symmetric matrices $\mathcal{S}^n$. A set $\mathcal{M}\subseteq\mathcal{S}^n$ is {\em orthogonally invariant} provided $U \mathcal{M} U^\top =\mathcal{M}$ for all matrices $U\in O(n)$. Such a set $\mathcal{M}$ can be written as 
$\lambda^{-1}(S)$ where $\lambda\colon\mathcal{S}^n\to\RR^n$ assigns to each matrix $X$ the vector of its eigenvalues in a nonincreasing order and $S$ is the diagonal restriction
$S=\{x\in\RR^n: \Diag(x)\in \mathcal{M}\}$. Conversely any permutation invariant set $S\subseteq\RR^n$ gives rise to the orthogonally invariant set $\lambda^{-1}(S)$. Similar techniques to the ones developed here can then be used to study the correspondence between ED critical points of algebraic varieties $S$ and $\lambda^{-1}(S)$. 
This research direction deserves further investigation. }
\end{remark}

\bibliographystyle{plain}
\bibliography{bibliography}
\end{document}